\documentclass{article}
\usepackage{amsmath, amsthm}
\usepackage{amssymb}
\newtheoremstyle{theorem}
  {10pt}          
  {10pt}  
  {\sl}  
  {\parindent}     
  {\bf}  
  {. }    
  { }    
  {}     
\newtheorem{thm}{Theorem}[section]
\newtheorem{cor}[thm]{Corollary}
\newtheorem{lem}[thm]{Lemma}

\newtheorem{defn}{Definition}[section]

\newtheorem{rem}{Remark}[section]
\begin{document}
\title{\large\bf A monotonicity result  for the  $q-$fractional operator}
\author{\small \bf Bahaaeldin Abdalla $^{a}$,  T. Abdeljawad $^{a}$, Juan J. Nieto $^{c}$'$^{d}$   \\ {\footnotesize $^a$ Department of Mathematics and General Sciences,
 Prince Sultan University-Riyadh-KSA}\\ {\footnotesize $^c$ Faculty of Science, King Abdulaziz University } \\ {\footnotesize $^d$ Facultade de Matemáticas, Universidade de Santiago de Compostela } }
\date{}
\maketitle {\footnotesize {\noindent\bf Abstract.}  In this article we prove that if the $q-$fractional operator  $(~_{q}\nabla_{qa}^\alpha y)(t)$ of order $0<\alpha\leq 1$ , $0<q<1$ and starting at some $qa \in T_q=\{q^k: k \in \mathbb{Z}\}\cup \{0\},~~a>0$ is positive such that $y(a) \geq 0$,  then $y(t)$ is $c_q(\alpha)-$increasing, $c_q(\alpha)=\frac{1-q^\alpha}{1-q}q^{1-\alpha}$. Conversely, if y(t) is increasing and $y(a)\geq 0$, then $(~_{q}\nabla_{qa}^\alpha y)(t)\geq 0$. As an application, we proved a $q-$fractional version of the Mean-Value  Theorem .}\\ 


\section{Introduction and Preliminaries}

Fractional calculus \cite{Podlubny, Samko, Kilbas} has recently occupied the minds of many researchers either  theoretically or in different fields of applications  \cite{r14, FTH}. The theory of $q-$fractional calculus was initiated in early of fifties of last century \cite{Hahn, Agarwal, Alsalam, Alsalam1, Alsalam2}. Then, this theory has started to be  developed in the last decade or so \cite{Pred}-\cite{TFD q}. For the preliminaries about $q-$ fractional calculus given here shortly, we refer the reader to the survey \cite{history} and the recent book \cite{Annaby}. On the other hand the theory of dicrete fractional calculus started to develop rapidly specially in the last decade \cite{Th}-\cite{r12}. Very recently, some monotonicity results have been reported for fractional difference type operators of order $0 <\alpha \leq 1$ \cite{Ferhan Monot}, and of order $1< \alpha <2$ \cite{Baoguo, Goodrich}. Motivated by what mentioned above and the fact that monotonicity results are of interest in usual calculus itself we obtain some monotonicity results for the $q-$fractional type operators of order $0<\alpha \leq 1$ in Section 2. An application is also given in Section 3 by giving a $q-$fractional mean value theorem version.

For $0<q<1,$ let $T_{q}$ be the time scale $$T_{q}=\{q^{n}:n\in \mathbb{Z}\}\cup \{0\}$$
More generally, if $\alpha$ is a nonnegative real number then we define the time scale
$$T_{q}^{\alpha}=\{q^{n+\alpha}:n \in \mathbb{Z}\}\cup \{0\}$$
We write $T_{q}^{0}=T_{q}.$\\
For a function $f:T_{q}\rightarrow \mathbb{R},$ the nabla $q-$derivative of $f$ is given by
\begin{equation}
\nabla_{q}f(t)=\frac{f(t)-f(qt)}{(1-q)t}, \; t \in T_{q} - \{0\}
\end{equation}
The nabla $q-$integral of $f$ is given by
\begin{equation}
\int_{0}^{t}{f(s)\nabla_{q}s}=(1-q)t\sum_{i=0}^{\infty}{q^{i}f(tq^{i})}
\end{equation}
and for $0\leq a \in T_{q}$
$$\int_{a}^{t}{f(s)\nabla_{q}s}=\int_{0}^{t}{f(s)\nabla_{q}s}-\int_{0}^{a}{f(s)\nabla_{q}s}$$
Alternatively, let $a=q^{n_0} \in T_q$ where $n<n_0$, the nabla $q-$integral of $f$ is given by
\begin{equation}
\int_{a}^{t}{f(s)\nabla_{q}s}=(1-q)\sum_{i=n}^{n_{0}-1}{q^{i}f(q^{i})}
\end{equation}
By the fundamental theorem in $q-$calculus we have
\begin{equation}
\nabla_{q} \int_{0}^{t}{f(s)\nabla_{q}s}=f(t)
\end{equation}
and if $f$ is continuous at 0, then
\begin{equation}
\int_{0}^{t}{\nabla_{q}f(s)\nabla_{q}s}=f(t)-f(0)
\end{equation}
Also the following identity will be helpful
\begin{equation}
\nabla_{q} \int_{a}^{t}{f(t,s)\nabla_{q}s}= \int_{a}^{t}{\nabla_{q}f(t,s)\nabla_{q}s}+f(qt,t)
\end{equation}
From the theory of $q-$calculus and the theory of time scale more generally, the following product rule is valid
\begin{equation}
\nabla_{q}(f(t)g(t))=f(qt)\nabla_{q}g(t)+(\nabla_{q}f(t))g(t)
\end{equation}
The $q-$factorial function for $n \in \mathbb{N}$ is defined by
\begin{equation}
(t-s)_{q}^{n}=\prod_{i=0}^{n-1}{(t-q^{i}s)}
\end{equation}
More generally, when $\alpha$ is a not a positive integer, the $q-$factorial fractional function is defined by
\begin{equation} \label{6}
(t-s)_{q}^{\alpha}=t^{\alpha} \prod_{i=0}^{\infty}{\frac{1-\frac{s}{t}q^{i}}{1-\frac{s}{t}q^{i+\alpha}}}
\end{equation}
It has the following properties
\begin{itemize}
  \item $(t-s)_{q}^{\beta +\gamma}=(t-s)_{q}^{\beta}(t-q^{\beta}s)_{q}^{\gamma}$
 \item $(at-as)_{q}^{\beta}=a^{\beta}(t-s)_{q}^{\beta}$
  \item The nabla $q-$derivative of the $q-$factorial function with respect to $t$ is $$\nabla_{q}(t-s)_{q}^{\alpha}=\frac{1-q^{\alpha}}{1-q}(t-s)_{q}^{\alpha-1}$$
 \item The nabla $q-$derivative of the $q-$factorial function with respect to $s$ is $$\nabla_{q}(t-s)_{q}^{\alpha}=-\frac{1-q^{\alpha}}{1-q}(t-qs)_{q}^{\alpha-1}$$ where $\alpha , \beta , \gamma \in \mathbb{R}.$
\end{itemize}
Moreover, the $q-$fractional integral of order $\alpha \neq 0,-1,-2,...$ is defined by
\begin{equation}
_{q}I_{0}^{\alpha}f(t)=\frac{1}{\Gamma_{q}(\alpha)} \int_{0}^{t}{(t-qs)_{q}^{\alpha -1}f(s)\nabla_{q}s}.
\end{equation}
Let $\alpha >0.$ If $\alpha \notin \mathbb{N},$ then the $\alpha -$order Caputo (left) $q-$fractional derivative of a function $f$ is defined by \cite{TD2011}
\begin{equation}
_{q}C_{a}^{\alpha}f(t)\triangleq \: _{q}I_{a}^{(n-\alpha)}\nabla_{q}^{n}f(t)=\frac{1}{\Gamma (n-\alpha)}\int_{a}^{t}{(t-qs)_{q}^{n-\alpha -1}\nabla_{q}^{n}f(s)\nabla_{q}s}
\end{equation}
where $n=[\alpha]+1$ and $[\alpha]$ denotes the greatest integer less than or equal to $\alpha$.
If $\alpha \in \mathbb{N},$ then $_{q}C_{a}^{\alpha}f(t)\triangleq \nabla_{q}^{n}f(t)$\\
The following identity is useful to transform Caputo $q-$fractional difference  equation into $q-$fractional integrals.\\
Assume $\alpha >0$ and $f$ is defined in suitable domains. Then \cite{TD2011}
\begin{equation}
_{q}I_{a}^{\alpha} \, _{q}C_{a}^{\alpha}f(t)=f(t)-\sum_{k=0}^{n-1}{\frac{(t-a)_{q}^{k}}{\Gamma_{q}(k+1)}\nabla_{q}^{k}f(a)}
\end{equation}
and if $0<\alpha \leq 1$ then
\begin{equation}
_{q}I_{a}^{\alpha} \: _{q}C_{a}^{\alpha}f(t)=f(t)-f(a)
\end{equation}
The following identity is essential to solve linear $q-$fractional equations
\begin{equation}
_{q}I_{a}^{\alpha}(x-a)_{q}^{\mu}=\frac{\Gamma_{q}(\mu +1)}{\Gamma_{q}(\alpha + \mu +1)}(x-a)_{q}^{\mu + \alpha}\;\; (0\leq a<x<b)
\end{equation}
where $\alpha \in \mathbb{R}^{+}$ and $\mu \in (-1,\infty).$\\
For more about $q-$Gamma functions and other  $q-$calculus concepts we refer, for example.  to \cite{history}
The following lemma is used in the proof of the main result. The proof follows from (\ref{6}).
\begin {lem}
\begin{itemize}
 \item $(1-q^i)_{q}^{-\alpha}=\frac{1-q^i}{1-q^{i-\alpha}}(1-q^{i+1})_{q}^{-\alpha}$
 \item $(q^{n+1}-1)_{q}^{-\alpha}=\frac{1-q^{-1-n}}{q^{\alpha}-q^{-1-n}}(q^n-1)_{q}^{-\alpha}$
 \item $(q^m-q^n)_{q}^{-\alpha}=\frac{1-q^{n-m}}{q^{\alpha}-q^{n-m}}(q^{m-1}-q^n)_{q}^{-\alpha}$
 \item $(q^m-q^{n-1})_{q}^{-\alpha}=\frac{1-q^{-m+n-1}}{1-q^{-m+n-1-\alpha}}(q^{m}-q^n)_{q}^{-\alpha}$
\end{itemize}
\end{lem}
\begin{defn}
 Fix $\alpha \geq 0$ and define $$c_q(\alpha)= \frac {1-q^{\alpha}}{1-q}q^{1-\alpha}$$
 \end{defn}

\begin{defn}
Let $y:T_q \rightarrow  {R} $ be a function. y is called a $c_q(\alpha)-$increasing on $T_q$, if
  $$y(q^{n_-1}) \geq c_q(\alpha)y(q^{n}) \; \mbox{for all} \; q^n \in T_q.$$
 \end{defn}
 \begin{defn}
Let $y:T_q \rightarrow  {R} $ be a function. y is called a $c_q(\alpha)-$decreasing on $T_q$, if
  $$y(q^{n_-1}) \leq c_q(\alpha)y(q^{n}) \; \mbox{for all} \; q^n \in T_q.$$
 \end{defn}
  Notice that if $\alpha \geq 1$, then $c_q(\alpha) \geq 1$ and if $0\leq \alpha \leq 1$, then $0\leq c_q(\alpha) \leq 1$. Hence, if $y$ is increasing (decreasing) on $T_q$ and and $0<\alpha <1$ then $y$ is $c_q(\alpha)-$increasing (decreasing) on $T_q$. Also, if $y$ is $c_q(\alpha)-$increasing (decreasing) on $T_q$ and $\alpha >1$, then, $y$ is increasing (decreasing) on $T_q$. If $\alpha=1$ then $y$ is $c_q(\alpha)-$increasing (decreasing) if and only if $y$ is increasing (decreasing).
\section{Main Results}
\begin{thm} \label{first}
Let $y:T_q \rightarrow  \mathbb{R} $ be a function satisfying $y(a) \geq 0$, $a=q^{n_0}>0$. Suppose that
 $$_{q}\nabla_{aq}^{\alpha} y(t)\geq 0 \mbox{ for each}~ \; t=q^n, n<n_0.$$
 Then, $y$ is $c_q(\alpha)-$increasing on $\{t \in T_q: t\geq a = q^{n_0}\}$.
\end{thm}
\begin{proof}
Let $_{q}\nabla_{aq}^{\alpha} y(t)\geq 0 $  for each  $t \in T_q$, $\alpha \in (0,1)$, then
$$ _{q}\nabla_{aq}^{\alpha} y(t)=\nabla _{q}\nabla _{aq}^{-(1-\alpha)} y(t)=\nabla_q \left[ \frac{1}{\Gamma_q (1-\alpha)}\int_{aq}^{t}(t-qs)_{q}^{-\alpha}y(s)\nabla_q (s)\right] \geq 0$$
Let  $s(t)=\frac{1-q}{\Gamma_q (1-\alpha)} \sum_{i=n_0}^n q^i (q^n-q^{i+1})_q^{-\alpha}y(q^i).$ Since $\nabla_q s(t) \geq 0,$ $s(t)$ is an increasing function on $T_q.$ This implies that
\begin{eqnarray*}
s(q^{n_0-1})-s(q^{n_0})
&=& \frac{(1-q)q^{n_0-1}}{\Gamma_q(1-\alpha)}[(q^{n_0-1}-q^{n_0})_q^{-\alpha}y(q^{n_0-1})\\
&+& q(q^{n_0-1}-q^{n_0+1})_q^{-\alpha}y(q^{n_0})-q(q^{n_0}-q^{n_0+1})_q^{-\alpha}y(q^{n_0})]\\
&=& \frac{(1-q)q^{n_0-1}}{\Gamma_q(1-\alpha)}[q^{\alpha}(q^{n_0}-q^{n_0+1})_q^{-\alpha}y(q^{n_0-1})\\
&+&q\frac{q^{\alpha}-q}{1-q}(q^{n_0}-q^{n_0+1})_q^{-\alpha}y(q^{n_0})-q(q^{n_0}-q^{n_0+1})_q^{-\alpha}y(q^{n_0})]\\
&\geq&0
\end{eqnarray*}
therefore, we have
$$q[\frac{q^{\alpha}-q}{1-q}-1]y(q^{n_0})+q^{\alpha}y(q^{n_0-1})\geq 0$$
which implies that
$$y(q^{n_0-1})\geq \frac{1-q^{\alpha}}{1-q}q^{1-\alpha}y(q^{n_0})$$
Now, we assume that the hypothesis is true for $n=k$. i.e.
$$y(q^{n_0-k})\geq \frac{1-q^{\alpha}}{1-q}q^{1-\alpha}y(q^{n_0-k+1})$$
hence,we have
\begin{equation} \label{7}
y(q^{n_0-k})\geq c_q(\alpha)y(q^{n_0-k+1})\geq c_{q}^2(\alpha)y(q^{n_0-k+2})\geq ...\geq c_{q}^{k-1}(\alpha)y(q^{n_0-1})\geq c_{q}^{k}(\alpha)y(q^{n_0})
\end{equation}
We want to prove that
\begin{equation} \label{8}
y(q^{n_0-k-1})\geq c_q(\alpha)y(q^{n_0-k}).
\end{equation}
We start by calculating,
\begin{eqnarray*}
&& s(q^{n_0-k-1})-s(q^{n_0-k})\\
&=&\frac{(1-q)}{\Gamma_q(1-\alpha)}\left[\sum_{i=n_0-k-1}^{n_0} q^i (q^{n_0-k-1}-q^{i+1})_q^{-\alpha}y(q^i)-\sum_{i=n_0-k}^{n_0} q^i (q^{n_0-k}-q^{i+1})_q^{-\alpha}y(q^i)\right]\\
&=&\frac{(1-q)}{\Gamma_q(1-\alpha)}[q^{n_0-k-1} (q^{n_0-k-1}-q^{n_0-k})_q^{-\alpha}y(q^{n_0-k-1})\\
&+&q^{n_0-k} (q^{n_0-k-1}-q^{n_0-k+1})_q^{-\alpha}y(q^{n_0-k})-q^{n_0-k} (q^{n_0-k}-q^{n_0-k+1})_q^{-\alpha}y(q^{n_0-k})\\
&+&q^{n_0-k+1} (q^{n_0-k-1}-q^{n_0-k+2})_q^{-\alpha}y(q^{n_0-k+1})-q^{n_0-k+1}(q^{n_0-k}-q^{n_0-k+2})_q^{-\alpha}y(q^{n_0-k+1})\\
&+&.....\\
&+&q^{n_0-1} (q^{n_0-k-1}-q^{n_0})_q^{-\alpha}y(q^{n_0-1})-q^{n_0-1} (q^{n_0-k}-q^{n_0})_q^{-\alpha}y(q^{n_0-1})\\
&+&q^{n_0} (q^{n_0-k-1}-q^{n_0+1})_q^{-\alpha}y(q^{n_0})-q^{n_0} (q^{n_0-k}-q^{n_0+1})_q^{-\alpha}y(q^{n_0})]\\
&=&\frac{(1-q)}{\Gamma_q(1-\alpha)}[q^{n_0-k-1} (q^{n_0-k-1}-q^{n_0-k})_q^{-\alpha}y(q^{n_0-k-1})\\
&+&q^{n_0-k}(q^{n_0-k}-q^{n_0-k+1})_q^{-\alpha}y(q^{n_0-k})(\frac{q^{\alpha}-q}{1-q}-1)\\
&+&q^{n_0-k+1}(q^{n_0-k}-q^{n_0-k+2})_q^{-\alpha}y(q^{n_0-k+1})(\frac{q^{\alpha}-q^2}{1-q^2}-1)\\
&+&.....\\
&+&q^{n_0-1}(q^{n_0-k}-q^{n_0})_q^{-\alpha}y(q^{n_0-1})(\frac{q^{\alpha}-q^{k}}{1-q^{k}}-1)\\
&+&q^{n_0}(q^{n_0-k}-q^{n_0+1})_q^{-\alpha}y(q^{n_0})(\frac{q^{\alpha}-q^{k+1}}{1-q^{k+1}}-1)]\\
\end{eqnarray*}
Since $s(t)$ is increasing, we get
\begin{eqnarray*}
&&q^{n_0-k-1} (q^{n_0-k-1}-q^{n_0-k})_q^{-\alpha}y(q^{n_0-k-1})+q^{n_0-k}(q^{n_0-k}-q^{n_0-k+1})_q^{-\alpha}y(q^{n_0-k})(\frac{q^{\alpha}-1}{1-q})\\
&\geq&q^{n_0-k+1}(q^{n_0-k}-q^{n_0-k+2})_q^{-\alpha}y(q^{n_0-k+1})(\frac{1-q^{\alpha}}{1-q^2})\\
&+&.....\\
&+&q^{n_0-1}(q^{n_0-k}-q^{n_0})_q^{-\alpha}y(q^{n_0-1})(\frac{1-q^{\alpha}}{1-q^{k}})\\
&+&q^{n_0}(q^{n_0-k}-q^{n_0+1})_q^{-\alpha}y(q^{n_0})(\frac{1-q^{\alpha}}{1-q^{k+1}})\\
\end{eqnarray*}
Using the induction assumption (\ref{7}), we get
\begin{eqnarray*}
&&q^{n_0-k-1}(q^{n_0-k}-q^{n_0-k+1})_q^{-\alpha}\left[q^{\alpha}y(q^{n_0-k-1})+q\frac{q^{\alpha}-1}{1-q}y(q^{n_0-k})\right]\\
&\geq&q^{n_0-k+1}(q^{n_0-k}-q^{n_0-k+2})_q^{-\alpha}(\frac{1-q^{\alpha}}{1-q^2})(c_q(\alpha))^{k-1}y(q^{n_0})\\
&+&.....\\
&+&q^{n_0-1}(q^{n_0-k}-q^{n_0})_q^{-\alpha}(\frac{1-q^{\alpha}}{1-q^{k}})(c_q(\alpha))y(q^{n_0})\\
&+&q^{n_0}(q^{n_0-k}-q^{n_0+1})_q^{-\alpha}(\frac{1-q^{\alpha}}{1-q^{k+1}})y(q^{n_0})\\
\end{eqnarray*}
Since $y(q^{n_0})\geq0$, we conclude that $q^{\alpha}y(q^{n_0-k-1})+q\frac{q^{\alpha}-1}{1-q}y(q^{n_0-k})\geq0$ which implies that $y(q^{n_0-k-1})\geq c_q(\alpha)y(q^{n_0-k})$
\end{proof}

Using Theorem \label{first} above and  the following identity that relates (Riemann) $q-$fractional derivative $~_{q}\nabla_a^\alpha $ and the Caputo $q-$fractional derivative $~_{q}C_a^\alpha$ of order $0<\alpha <1$ \cite{Th}

$$(~_{q}C_a^\alpha f)(t)=(~_{q}\nabla_a^\alpha f)(t)-\frac{(t-a)_q^{-\alpha}}{\Gamma_q(1-\alpha)}  y(a), $$

we can state the following Caputo monotonicity result:

\begin{cor}
Let $y:T_q \rightarrow  \mathbb{R} $ be a function satisfying $y(a) \geq 0$, $a=q^{n_0}>0$. Suppose that
 $$_{q}C_{aq}^{\alpha} y(t)\geq - \frac{(t-qa)_q^{-\alpha}}{\Gamma_q(1-\alpha)}  y(qa) \mbox{ for each}~ \; t=q^n, n<n_0.$$
 Then, $y$ is $c_q(\alpha)-$increasing on $\{t \in T_q: t\geq a = q^{n_0}\}$.
\end{cor}

\begin{thm}
Let $y:T_q \rightarrow  \mathbb{R} $ be a function satisfying $y(q^{n_0}) \geq 0$, $a=q^{n_0}$. Suppose that $y$ is an increasing function on $T_q$. Then,
 $$_{q}\nabla_{aq}^{\alpha} y(t)\geq 0 \mbox{ for each} \; t=q^n, n<n_0.$$
\end{thm}
\begin{proof}
We want to prove that
$$ _{q}\nabla_{aq}^{\alpha} y(t)=\nabla _{q}\nabla _{aq}^{-(1-\alpha)} y(t)=\nabla_q \left[ \frac{1}{\Gamma_q (1-\alpha)}\int_{aq}^{t}(t-qs)_{q}^{-\alpha}y(s)\nabla_q (s)\right] \geq 0$$
Let $$s(t)=\left[ \frac{1}{\Gamma_q (1-\alpha)}\int_{aq}^{t}(t-qs)_{q}^{-\alpha}y(s)\nabla_q (s)\right]=\frac{1-q}{\Gamma_q (1-\alpha)} \sum_{i=n_0}^n q^i (q^n-q^{i+1})_q^{-\alpha}y(q^i).$$ Since $\nabla_q s(t) \geq 0.$
We need to show that $s(t)$ is increasing on $T_q.$ i.e. we need to show that
$s(q^{n_0-k-1})-s(q^{n_0-k})\geq0$ for any natural number $k$ with $k\geq1.$
In fact,
\begin{eqnarray*}
&& s(q^{n_0-k-1})-s(q^{n_0-k})\\
&=&\frac{(1-q)}{\Gamma_q(1-\alpha)}\left[\sum_{i=n_0-k-1}^{n_0} q^i (q^{n_0-k-1}-q^{i+1})_q^{-\alpha}y(q^i)-\sum_{i=n_0-k}^{n_0} q^i (q^{n_0-k}-q^{i+1})_q^{-\alpha}y(q^i)\right]\\
&=&\frac{(1-q)}{\Gamma_q(1-\alpha)}[q^{n_0-k-1} (q^{n_0-k-1}-q^{n_0-k})_q^{-\alpha}y(q^{n_0-k-1})\\
&+&\sum_{i=n_0-k}^{n_0} q^i (q^{n_0-k-1}-q^{i+1})_q^{-\alpha}y(q^i)-\sum_{i=n_0-k}^{n_0} q^i (q^{n_0-k}-q^{i+1})_q^{-\alpha}y(q^i)]\\
&=&\frac{(1-q)}{\Gamma_q(1-\alpha)}[q^{n_0-k-1}(q^{n_0-k-1}-q^{n_0-k})_q^{-\alpha}y(q^{n_0-k-1})\\
&+&\sum_{i=n_0-k}^{n_0} q^i [\frac{q^{\alpha}-q^{i+1-n_0+k}}{1-q^{i+1-n_0+k}}-1](q^{n_0-k}-q^{i+1})_q^{-\alpha}y(q^i)]\\
&=&\frac{(1-q)}{\Gamma_q(1-\alpha)}[q^{n_0-k-1}(q^{n_0-k-1}-q^{n_0-k})_q^{-\alpha}y(q^{n_0-k-1})\\
&+&\sum_{i=n_0-k}^{n_0} q^i \frac{q^{\alpha}-1}{1-q^{i+1-n_0+k}}(q^{n_0-k}-q^{i+1})_q^{-\alpha}y(q^{n_0-k-1}-q^i)\\
&-&\sum_{i=n_0-k}^{n_0} q^i \frac{q^{\alpha}-1}{1-q^{i+1-n_0+k}}(q^{n_0-k}-q^{i+1})_q^{-\alpha}y(q^{n_0-k-1})]\\
&\geq&\frac{(1-q)}{\Gamma_q(1-\alpha)}[q^{n_0-k-1}(q^{n_0-k-1}-q^{n_0-k})_q^{-\alpha}y(q^{n_0-k-1})\\
&+&\sum_{i=n_0-k}^{n_0} q^i \frac{1-q^{\alpha}}{1-q^{i+1-n_0+k}}(q^{n_0-k}-q^{i+1})_q^{-\alpha}y(q^{n_0-k-1})]\geq0.
\end{eqnarray*}
\end{proof}
\begin{thm}\label{9}
Let $y:T_q \rightarrow  \mathbb{R} $ be a function satisfying $y(q^{n_0}) > 0$, $a=q^{n_0}$. Suppose that $y$ is a strictly increasing function on $T_q$. Then,
 $$_{q}\nabla_{aq}^{\alpha} y(t)> 0 \mbox{ for each} \; t=q^n, n<n_0.$$
\end{thm}
In a similar way, the above results can be obtained for the function which takes negative value at the initial point of its domain.
\begin{thm}
Let $y:T_q \rightarrow  \mathbb{R} $ be a function satisfying $y(q^{n_0}) \leq 0$. Suppose that
 $$_{q}\nabla_{aq}^{\alpha} y(t)\leq 0 \mbox{ for each} \; t=q^n, n<n_0.$$
 Then, $y$ is $c_q(\alpha)-$decreasing on $T_q$
\end{thm}
\begin{thm}
Let $y:T_q \rightarrow  \mathbb{R} $ be a function satisfying $y(q^{n_0}) \leq 0$. Suppose that $y$ is a decreasing function on $T_q$. Then,
 $$_{q}\nabla_{aq}^{\alpha} y(t)\leq 0 \mbox{ for each}~ \; t=q^n, n<n_0.$$
\end{thm}
\section{An Application}
In this section, we wish to prove a Mean-Value Theorem in $q-$fractional calculus. First, we need the following result:
\begin{thm} \label{Tool}
Let $f$ be defined on $T_q$ and $a,b \in T_q$ with $a<b.$ Then the following equality holds:
$$_{q}\nabla_{a}^{-\alpha} \: _{q}\nabla_{aq}^{\alpha} f(t)|_b = f(b)- \frac{(1-q)a^{1-\alpha}(b-a)_q^{\alpha -1}(1-q)_q^{-\alpha }}{\Gamma _q (\alpha)\Gamma _q (1-\alpha)}f(a)\:\: \mbox{where} \; \alpha \in (0,1).$$
\end{thm}
\begin{proof}
\begin{eqnarray*}
&& _{q}\nabla_{a}^{-\alpha} \: _{q}\nabla_{aq}^{\alpha} f(t)\left. \right|_b
=\:_{q}\nabla_{a}^{-\alpha} (\nabla_q \:_{q}\nabla_{aq}^{-(1-\alpha)} f(t))\left. \right|_b\\
&=&\nabla_q \: _{q}\nabla_{a}^{-\alpha} \:_{q}\nabla_{aq}^{-(1-\alpha)} f(t)\left|_b \right.-\frac{(t-a)_q^{\alpha -1}}{\Gamma_q(\alpha)}\: _{q}\nabla_{aq}^{-(1-\alpha)}f(a)\left. \right|_b\\
&=&\nabla_q \: _{q}\nabla_{a}^{-\alpha} \left[_{q}\nabla_{a}^{-(1-\alpha)} f(t)+\frac{(1-q)a(t-qa)_q^{-\alpha}f(a)}{\Gamma _q (1-\alpha)}\right]\left. \right|_b -\frac{(t-a)_q^{\alpha -1}}{\Gamma_q(\alpha)}\: _{q}\nabla_{aq}^{-(1-\alpha)}f(a)\left. \right|_b\\
&=&\nabla_q \: _{q}\nabla_{a}^{-1} f(t)\left. \right|_b -\frac{(t-a)_q^{\alpha -1}}{\Gamma_q(\alpha)}\: _{q}\nabla_{aq}^{-(1-\alpha)}f(a)\left. \right|_b\\
&=&f(b)-\frac{(1-q)a^{1-\alpha}(b-a)_q^{\alpha -1}(1-q)_q^{-\alpha }}{\Gamma _q (\alpha)\Gamma _q (1-\alpha)}f(a).
\end{eqnarray*}
\end{proof}
Let $M_q(\alpha,a,b)=\frac{(1-q)a^{1-\alpha}(b-a)_q^{\alpha -1}(1-q)_q^{-\alpha }}{\Gamma _q (\alpha)\Gamma _q (1-\alpha)}$.
\begin{thm} \label{MVT}
Assume $f$ and $g$ are defined on $T_q$, $a, b \in T_q,~~a>0$ and $g$ is strictly increasing on $\Delta = \left\{t\in T_q : t \geq a, t \leq b \right\}$ and satisfying $g(a)>0$. Then, there exists $r_1,r_2 \in \Delta$ such that
$$\frac{_{q}\nabla_{aq}^{\alpha}f(r_1)}{_{q}\nabla_{aq}^{\alpha}g(r_1)} \leq \frac{f(b)-M_q(\alpha,a,b)f(a)}{g(b)-M_q(\alpha,a,b)g(a)}\leq \frac{_{q}\nabla_{aq}^{\alpha}f(r_2)}{_{q}\nabla_{aq}^{\alpha}g(r_2)}.$$
\end{thm}
\begin{proof}
Suppose without loss of generality to the contrary that
$$\frac{f(b)-M_q(\alpha,a,b)f(a)}{g(b)-M_q(\alpha,a,b)g(a)}> \frac{_{q}\nabla_{aq}^{\alpha}f(t)}{_{q}\nabla_{aq}^{\alpha}g(t)} $$
Since $g$ is strictly increasing, Theorem (\ref{9}) implies that $_{q}\nabla_{aq}^{\alpha}g(t)> 0$. Hence
$$\frac{f(b)-M_q(\alpha,a,b)f(a)}{g(b)-M_q(\alpha,a,b)g(a)}\: _{q}\nabla_{aq}^{\alpha}g(t)>\: _{q}\nabla_{aq}^{\alpha}f(t).$$
We apply $_{q}\nabla_{a}^{-\alpha}$ at $t=b$ to both sides of the above inequality and make use of Theorem \ref{Tool} to get
$$\frac{f(b)-M_q(\alpha,a,b)f(a)}{g(b)-M_q(\alpha,a,b)g(a)}\:[g(b)-M_q(\alpha,a,b)g(a)]> f(b)-M_q(\alpha,a,b)f(a).$$ This leads to $g(b)> g(b)$,  which is a contradiction.
\end{proof}

\begin{rem}
\begin{itemize}
  \item Notice that if the edge point $a=0$ then $M_q(\alpha,a,b)=0$ and hence no Mean-Value Theorem.
  \item Notice that since $M_q(\alpha,a,b)<1$ and $g(t)$ is strictly decreasing then $g(b)-M_q(\alpha,a,b)g(a)$ in the statement of Theorem \ref{MVT} is not equal to zero.
\end{itemize}
\end{rem}

\end{document}